\documentclass[12pt]{article}

\usepackage{mathtools}
\usepackage{amssymb}
\usepackage{amsthm}
\usepackage{xcolor}
\usepackage{tikz}

\DeclareSymbolFont{bbold}{U}{bbold}{m}{n}
\DeclareSymbolFontAlphabet{\mathbbold}{bbold}

\newcommand{\N}{\mathbb{N}}

\newcommand{\R}{\mathbb{R}}
\newcommand{\C}{\mathbb{C}}
\newcommand{\K}{\mathbb{K}}
\newcommand{\1}{\mathbbold{1}}

\newcommand{\loc}{\mathrm{loc}}


\DeclareMathOperator{\grad}{grad}
\DeclareMathOperator{\diverg}{div}
\renewcommand{\div}{\diverg}

\DeclareMathOperator{\spt}{spt}

\DeclareMathOperator{\capacity}{cap}
\renewcommand{\Cap}{\capacity}

\providecommand{\form}{\tau}
\providecommand{\scpr}[2]{\left( #1 \,\middle|\, #2 \right)}
\renewcommand{\sp}{\scpr}
\newcommand{\from}{\colon}

\let\phi\varphi

\let\leq\leqslant

\let\geq\geqslant

\DeclareMathOperator{\tr}{tr}

\makeatletter
\def\@row#1,{#1\@ifnextchar;{\@gobble}{&\@row}}
\def\@matrix{%
    \expandafter\@row\my@arg,;%
    \@ifnextchar({\\ \get@in@paren{\@matrix}}{\after@matrix}%
    }
\def\matrixtype#1#2#3{%
    \ifmmode\def\after@matrix{\end{#2}\right#3}%
    \else\def\after@matrix{\end{#2}\right#3$}$\fi
    \left#1\begin{#2}\get@in@paren{\@matrix}%
    }
\def\@column#1,{#1\@ifnextchar;{\@gobble}{\\ \@column}}
\newcommand\vect{}
\def\svect(#1){\left(\begin{smallmatrix}\@column#1,;\end{smallmatrix}\right)}
\def\vect{\get@in@paren{\@vect}}
\def\@vect{\left(\begin{matrix}\expandafter\@column\my@arg,;\end{matrix}\right)}
\def\get@in@paren#1({\def\my@arg{}\def\my@rest{}\def\after@get{#1}\get@arg}
\let\e@a\expandafter
\def\get@arg#1){\e@a\kl@test\my@rest#1(;}
\def\kl@test#1(#2;{\e@a\def\e@a\my@arg\e@a{\my@arg#1}%
                   \ifx:#2:\let\my@exec\after@get
                   \else\let\my@exec\get@arg
                        \e@a\def\e@a\my@arg\e@a{\my@arg(}%
                        \def@rest#2;%
                   \fi\my@exec}
\def\def@rest#1(;{\def\my@rest{#1\kl@zu}}
\def\kl@zu{)}

\makeatletter
\newcommand\MyPairedDelimiter{%
  \@ifstar{\My@Paired@Delimiter{{}}}
          {\My@Paired@Delimiter{}}%
}
\newcommand\My@Paired@Delimiter[4]{%
  \newcommand#2{%
    \@ifstar{\start@PD{#1}{\delimitershortfall=-1sp}{#3}{#4}}
            {\start@PD{#1}{}{#3}{#4}}%
  }%
}
\newcommand\start@PD[5]{%
  #1\mathopen{\mathpalette\put@delim@helper{\put@delim{#2}{#3}{.}{#5}}}%
  #5%
  \mathclose{\mathpalette\put@delim@helper{\put@delim{#2}{.}{#4}{#5}}}%
}
\newcommand\put@delim@helper[2]{%
  \hbox{$\m@th\nulldelimiterspace=0pt #2#1$}%
}
\newcommand\put@delim[5]{%
  \setbox\z@\hbox{$\m@th#5{#4}$}%
  \setbox\tw@\null
  \ht\tw@\ht\z@ \dp\tw@\dp\z@
  #1#5%
  \left#2\box\tw@\right#3%
}

\makeatother
\MyPairedDelimiter*{\abs}{\lvert}{\rvert}
\MyPairedDelimiter*{\norm}{\lVert}{\rVert}
\MyPairedDelimiter{\set}{\{}{\}}

\theoremstyle{plain} 
\newtheorem{theorem}{Theorem}[section]

\newtheorem{proposition}[theorem]{Proposition}
\theoremstyle{definition}
\newtheorem{example}[theorem]{Example}
\newtheorem*{definition}{Definition}
\newtheorem{remark}[theorem]{Remark}

\usepackage{enumitem}

\setenumerate[1]{nolistsep} 
\setenumerate[2]{nolistsep} 

\begin{document}

\medmuskip=4mu plus 2mu minus 3mu
\thickmuskip=5mu plus 3mu minus 1mu
\belowdisplayshortskip=9pt plus 3pt minus 5pt

\title{Dirichlet forms for singular diffusion in higher dimensions}

\author{Uta Freiberg and Christian Seifert\footnote{corresponding author}}

\maketitle

\begin{abstract}
  We describe singular diffusion in bounded subsets $\Omega$ of $\R^n$ by form 
methods and characterize the associated operator. We also prove positivity 
and contractivity of the corresponding semigroup.
  This results in a description of a stochastic process moving according to classical diffusion in one part of $\Omega$,
  where jumps are allowed through the rest of $\Omega$.
  
\end{abstract}


\section{Introduction}

The aim of this paper is to present a treatment of multi-dimensional ``singular'' diffusion in the framework of Dirichlet forms.
Singular diffusion (sometimes called gap-diffusion) in one dimension goes back at least to Feller \cite{Feller1954} and has a long history, see e.g. \cite{LangerSchenk1990} and references therein.

To describe singular diffusion, we consider a suitable measure $\mu$ on an open and bounded subset $\Omega\subseteq\R^n$,
and let particles move in $\Omega$ according to ``Brownian motion'', where the particles may only be located in the support $\spt\mu$ of $\mu$.
Furthermore, the particles are accelerated or slowed down by the ``speed measure'' $\mu$.
If $\mu$ is supported only on a proper subset of $\Omega$, 
in terms of the stochastic process describing the motion of a particle this yields a time changed process (on $\spt\mu$), see \cite[Section 6.2]{FukushimaOshimaTakeda1994}.
In terms of the Dirichlet form we may also see that as a trace of the corresponding Dirichlet space \cite[Section 6.2]{FukushimaOshimaTakeda1994}.

We want to treat the evolution by constructing the corresponding Dirichlet form. 
Since the particles moving according to Brownian motion are only located in $\spt\mu$, we will interpret the classical Dirichlet form in $L_2(\Omega,\mu)$.
There is an abstract generating theorem to find generators associated to forms defined in different spaces in \cite{ArendtTerElst2012}; 
however, our approach is different in that we consider the form itself in the Hilbert space $L_2(\Omega,\mu)$ (where the generator should act in).
We will characterize the generating self-adjoint operator, and show that the corresponding $C_0$-semigroup is submarkovian.
The associated process is a jump-diffusion process, with a diffusion part on $\spt\mu$ and jumps through $\Omega\setminus\spt\mu$.

Such singular diffusions in one dimension and the form approach were described 
in 
\cite{SolomyakVerbitsky1995,Freiberg2003,Freiberg20042005,SeifertVoigt2011,
Seifert2009}, see also e.g.\ 
\cite{Ouhabaz2005} for form methods. As it turns out in one dimension, functions in the 
domain of the form (and hence also the operator) have to be affine on the 
complement of $\spt\mu$. Since in one dimension affine functions are exactly 
the harmonic functions, this will be the right condition occurring in higher 
dimensions.

In higher dimensions there are only few results in the 
literature, see \cite{NaimarkSolomyak1995,HuLauNgai2006, Seifert2009}, 
focussing on the construction of the operator (however under somewhat different
 assumptions; we will work with capacities).

In Section \ref{sec:DF} we describe the setup and interpret the classical Dirichlet form in $L_2(\Omega,\mu)$.
The generator is characterized in Section \ref{sec:operator}, where also properties of the associated semigroup are proven.
In Section \ref{sec:applications} we apply our result to two different situations. First we consider singular diffusion supported on a subset of codimension $1$.
Then we apply our results to diffusion on a fractal domain (we choose the Koch's snowflake here).

\section{Dirichlet forms for singular diffusion}
\label{sec:DF}

Let $\K\in\{\R,\C\}$ denote the field of scalars. We write $\lambda^n$ for the 
$n$-dimensional Lebesgue measure on $\R^n$.

Let $\Omega\subseteq \R^n$ be open and bounded. 
We define the classical Dirichlet form $\form_0$ on $\Omega$ by
\begin{align*}
  D(\form_0) & := W_{2,0}^1(\Omega),\\
  \form_0(u,v) & := \int_\Omega \grad u \cdot \overline{\grad v} \quad(u,v\in D(\form_0)).
\end{align*}
The corresponding form norm $\norm{\cdot}_{\form_0} := \bigl(\form_0(\cdot) + 
\norm{\cdot}_{L_2(\Omega,\lambda^n)}^2\bigr)^{1/2}$ is just the usual 
$W_2^1$-norm on $\Omega$,
where $\form_0(u):=\form_0(u,u)$. 

We will provide some notions from potential theory, which will be needed in the following.
For an open subset $V\subseteq\Omega$ we define 
\[\Cap(V):= \inf\set{\norm{u}_{\form_0}^2;\; u\in D(\form_0),\, u\geq 1 
\;\text{$\lambda^n$-a.e.\ on } V}.\]
For arbitrary $A\subseteq \Omega$ we set
\[\Cap(A):=\inf\set{\Cap(V);\; V\subseteq \Omega\,\text{open},\, A\subseteq V}.\]
Then $\Cap(A)$ is called the \emph{capacity} of $A$. We say that a property holds true \emph{quasi everywhere (q.e.)} if there exists $N\subseteq\Omega$ of zero capacity such that
the property is satisfied on $\Omega\setminus N$.

Let $(F_k)_{k\in\N}$ be a sequence of closed subsets of $\Omega$ satisfying $F_k\subseteq F_{k+1}$ for all $k\in\N$. Then $(F_k)$ is called a \emph{nest} if $\Cap(\Omega\setminus F_k)\to 0$.
If $(F_k)$ is a nest then we set
\[C((F_k)):= \set{u\from \Omega\to \K;\; u|_{F_k}\in C(F_k)\quad(k\in\N)}.\]
A function $u\from \Omega\to \K$ is said to be \emph{quasi-continuous} if there exists a nest $(F_k)$ such that $u\in C((F_k))$. 
Note that this is equivalent to saying that for any $\varepsilon>0$ there exists an open subset $U\subseteq \Omega$ such that $\Cap(U) <\varepsilon$ and $u|_{\Omega\setminus U} \in C(\Omega\setminus U)$.

\begin{proposition}[see {\cite[Theorem 2.1.3]{FukushimaOshimaTakeda1994}}]
  Every $u\in D(\form_0)$ admits a q.e.\ uniquely defined quasi-continuous representative $\tilde{u}$.
\end{proposition}

We set (writing $\mathcal{B}(\Omega)$ for the Borel subsets of $\Omega$)
\begin{align*}
  M_0(\Omega):=\bigl\{& \mu\from\mathcal{B}(\Omega)\to [0,\infty];\; \mu \;\text{$\sigma$-additive},\\ 
  & \mu(N) = 0 \;\text{for any Borel set $N\subseteq \Omega$ of zero capacity}\bigr\}.
\end{align*}

It is easy to see that $\mu\in M_0(\Omega)$ if $\mu$ is absolutely continuous with respect to the Lebesgue measure $\lambda^n(\cdot\cap \Omega)$ on $\Omega$.
As shown in \cite[Theorem 4.1]{BrascheExnerKuperinSeba1994}, also the 
$(n-1)$-dimensional Hausdorff measure on $(n-1)$-dimensional $C^1$-submanifolds
 of $\Omega$ belongs to $M_0(\Omega)$.

Let $\mu\in M_0(\Omega)$ be a finite measure and $U:=\Omega\setminus\spt \mu$. 
The measure $\mu$ may be considered as a ``speed measure''.
Furthermore, we may assume
\begin{equation}
\label{eq:assumption}
  W_{2,0}^1(U) = \set{u\in W_{2,0}^1(\Omega);\; \tilde{u} = 0\,\text{$\mu$-a.e.}},
\end{equation}
where $\tilde{u}$ is a quasi-continuous representative of $u$.
Note that ``$\subseteq$'' is trivial; 
however, ``$\supseteq$'' does not hold in general, as the following example due to J\"urgen Voigt \cite{Voigt2009p} shows.

\begin{example}
  We start with a claim: Let $n\geq 2$, $\varepsilon>0$ and $r_0>0$. 
  Then there exist $0<r<r'\leq r_0$ and $\varphi\in C_c^1(\R^n)$ such that $\spt \varphi\subseteq B(0,r')$, $\1_{B[0,r]} \leq \varphi\leq 1$ and $\norm{\varphi}_{2,1} \leq\varepsilon$.
  Here $B(y,\rho)$ and $B[y,\rho]$ denote the open and closed balls around $y$ with radius $\rho$, respectively.
  
  Let $B_+:=\set{x\in B(0,1);\;x_1>0}$. Using the claim there exist $(x^k)$ in $B_+$, $(r_k)$ and $(r_k')$ in $(0,\infty)$ satisfying $r_k < r_k'$ for all $k\in\N$ and $(\varphi_k)$ in $C_c^1(\R^n)$ such that
  $\spt \varphi_k \subseteq B(x^k,r_k')$, $\1_{B[x^k,r_k]} \leq \varphi_k\leq 1$ such that
  \begin{itemize}
    \item the set of accumulation points of $(x^k)$ is exactly $\set{x\in B(0,1);\;x_1=0}$,
    \item $B(x^k,r_k')\cap B(x^j,r_j') = \varnothing$ for all $k,j\in\N$, $k\neq j$,
    \item $\sum_{k=1}^\infty \norm{\varphi_k}_{2,1}<\infty$.
  \end{itemize}
  Let $K:=\overline{\bigcup_{k\in\N} B[x^k,r_k']}$, $\Omega\supseteq K$ be open and bounded and $\mu$ the Lebesgue measure on $K$.
  Let $\varphi:=\sum_{k\in\N} \varphi_k$ and $\psi\in C_c^1(\R^n)$ such that $\psi = 1$ in a neighborhood of $K$. 
  Then $\psi-\varphi$ is quasi-continuous and $\psi-\varphi = 1$ on $\set{x\in B(0,1);\;x_1=0}$, a set with positive capacity.
  On the other hand, $\psi-\varphi = 0$ $\mu$-a.e., since $\psi -\varphi = 0$ on $\bigcup_{k\in\N} B[x^k,r_k]$ and the set 
  \[K\setminus \bigcup_{k\in\N} B[x^k,r_k] = \set{x\in B(0,1);\;x_1=0}\]
  has $\mu$-measure zero. Hence, $\psi-\varphi \in \set{u\in W_{2,0}^1(\Omega);\;\tilde{u} = 0\;\text{$\mu$-a.e.}}$, but
  \[\psi-\varphi \notin \set{u\in W_{2,0}^1(\Omega);\;\tilde{u} = 0\;\text{q.e. on $K$}}.\]
  By \cite[Theorem 1.13]{Hedberg1980} we observe
  \[\set{u\in W_{2,0}^1(\Omega);\;\tilde{u} = 0\;\text{q.e. on $K$}} = W_{2,0}^1(\Omega\setminus K).\]
  Thus, $\psi-\varphi\notin W_{2,0}^1(\Omega\setminus K)$.
\end{example}

\begin{remark}
\label{rem:general}
	In fact, we do not need condition \eqref{eq:assumption}. The subspace 
of $W_{2,0}^1(\Omega)$ we could work with is $\set{u\in W_{2,0}^1(\Omega);\;
 \tilde{u} = 0\,\text{$\mu$-a.e.}}^\bot$. As the following proposition shows, 
	in case \eqref{eq:assumption} is satisfied this subspace is exactly the 
space of $W_{2,0}^1(\Omega)$-functions, which are harmonic on 
$\Omega\setminus\spt\mu$.
\end{remark}

\begin{proposition}
  Let $\Omega\subseteq\R^n$ be open and bounded, $U\subseteq \Omega$ open.
  Then $W_{2,0}^1(\Omega) = W_{2,0}^1(U)\oplus D_{2,0}^1(U)$, where
  \[D_{2,0}^1(U):= \set{u\in W_{2,0}^1(\Omega);\; \Delta (u|_U) = 0}.\]
\end{proposition}

\begin{proof}
  Let $u\in W_{2,0}^1(\Omega)$. We show that there exists a unique $v\in W_{2,0}^1(U)$ such that
  \[0 = \int_U u\Delta \varphi - \int v\Delta \varphi \quad(\varphi\in C_c^\infty(U)).\]
  Then $Ju:=u-v\in D_{2,0}^1(U)$ and this implies the assertion.
  
  By Poincar\'{e}'s inequality we observe that
  \[(f,g)\mapsto \sp{f}{g}_0:=\int_U \grad f \cdot\overline{\grad g}\]
  defines an inner product on $W_{2,0}^1(U)$ such that this space becomes a Hilbert space.
  
  Since
  \[\abs{\int_U u\Delta\varphi} = \abs{\int_U \grad u \cdot \grad\varphi} \leq \norm{\abs{\grad u}}_{L_2(\Omega)} \norm{\varphi}_0 \quad(\varphi\in C_c^\infty(U)),\]
  the mapping $\varphi \mapsto -\int_U u \Delta\varphi$ is a continuous linear 
functional on $W_{2,0}^1(U)$. By Riesz' representation theorem there exists a
 unique $v\in W_{2,0}^1(U)$ such that
  \[\sp{\varphi}{\overline{v}}_0 = -\int_U u \Delta\varphi \quad(\varphi\in 
C_c^\infty(U)). \qedhere\]  
\end{proof}

Let $J\from W_{2,0}^1(\Omega)\to D_{2,0}^1(U)$ be the orthogonal projection. Then $\widetilde{Ju} = \tilde{u}$ $\mu$-a.e.

Let $D:=\set{u\in L_2(\Omega,\mu);\; \exists\, v\in W_{2,0}^1(\Omega): \tilde{v} = u\;\text{$\mu$-a.e.}}$.
Then $\iota\from D\to D_{2,0}^1(U)$, $\iota(u) := Jv$ is a well-defined linear mapping.

Define
\begin{align*}
  D(\form_D) & := D,\\
  \form_D(u,v) & := \int_\Omega \grad \iota(u)(x)\cdot \overline{\grad \iota(v)(x)}\, dx = \form_0(\iota(u),\iota(v)) \quad(u,v\in D(\form_D)).
\end{align*}

\begin{theorem}
  $\form_D$ is densely defined, symmetric, nonnegative and closed. Furthermore, $C_c^\infty(\Omega)$ is a core for $\form_D$.
\end{theorem}

\begin{proof}
  $\form_D$ is densely defined since $C_c^\infty(\Omega) \subseteq D(\form_D)$ is dense in $L_2(\Omega,\mu)$.
  Symmetry and non-negativity is clear by definition. 
  To show closedness, let $(u_n)$ in $D(\form_D)$ be a $\form_D$-Cauchy sequence, i.e.\ $\form_D(u_n-u_m)\to 0$, and $u_n\to u$ in $L_2(\Omega,\mu)$.
  By Poincar\'{e}'s inequality there exists $v\in W_{2,0}^1(\Omega)$ such that $\iota(u_n)\to v$ in $W_{2,0}^1(\Omega)$. For $\varphi\in C_c^\infty(\Omega)$ we compute
  \[0 = \int_U \iota(u_n) \Delta \varphi \to \int_U v\Delta\varphi,\]
  i.e.\ $v\in D_{2,0}^1(U)$.
  
  There exists a subsequence $(u_{n_k})$ such that $\widetilde{\iota(u_{n_k})}\to \tilde{v}$ q.e., and hence also $\mu$-a.e. 
  Since $\widetilde{\iota(u_n)} = u_n$ $\mu$-a.e.\ we observe $\tilde{v} = u$ $\mu$-a.e. Hence, $u\in D(\form_D)$ and
  \[\form_D(u_n-u) = \form_0(\iota(u_n)-v) \to 0.\]
  
  Let us now show that $C_c^\infty(\Omega)$ is a core for $\form_D$. 
  It suffices to approximate $0\leq u \in D(\form_D)$. First, assume that $u\in L_\infty(\mu)$.
  Then there exists a sequence $(\varphi_l)$ in $C_c^\infty(\Omega)$ such that $\varphi_l\to \iota(u)$ in $W_{2,0}^1(\Omega)$, 
  $\varphi_l\to \widetilde{\iota(u)}$ q.e.\ and $M:=\sup\set{\norm{\varphi_l}_{\infty,\spt\mu};\; l\in\N} <\infty$. 
  Since $\mu\in M_0(\Omega)$ we also have $\varphi_l\to \widetilde{\iota(u)}$ $\mu$-a.e., and since $\widetilde{\iota(u)} = u$ $\mu$-a.e.\ also $\varphi_l\to u$ $\mu$-a.e.
  Since $\abs{\varphi_l}\leq M\1_\Omega \in L_2(\mu)$ Lebesgue's dominated convergence theorem yields $\varphi_l\to u$ in $L_2(\mu)$, and therefore
  $\varphi_l\to u$ in $D_{\form_D} = (D(\form_D),\norm{\cdot}_{\form_D})$.
  
  For general $0\leq u \in D(\form_D)$ there exists $0\leq v\in W_{2,0}^1(\Omega)$ such that $\tilde{v} = u$ $\mu$-a.e. Then, for $k\in\N$, we have $u_k:=u\wedge k\in D(\form_D)$,
  where $f \wedge g:= \min\set{f,g}$ denotes the minimum,
  since $v\wedge k\in W_{2,0}^1(\Omega)$ and $\widetilde{v\wedge k} = \tilde{v}\wedge k = u_k$ $\mu$-a.e.
  Hence, for $k\in \N$ there exists $(\varphi_l^k)_l$ in $C_c^\infty(\Omega)$ such that $\varphi_l^k\to \iota(u_k)$ in $W_{2,0}^1(\Omega)$ and $\varphi_l^k\to u_k$ in $L_2(\mu)$.
  Since $u_k\to u$ in $L_2(\mu)$ and $v\wedge k\to v$ in $W_{2,0}^1(\Omega)$ we
 also have $\iota(u_k) = J(v\wedge k)\to Jv = \iota(u)$ in 
$W_{2,0}^1(\Omega)$. Hence, $u_k\to u$ in $D_{\form_D}$.
  Thus, a suitable subsequence of $(\varphi_l^k)$ converges to $u$ in $D_{\form_D}$.
\end{proof}

\section{Characterization of the operator}
\label{sec:operator}

Let $H$ be the self-adjoint operator in $L_2(\Omega,\mu)$ associated with $\form_D$, where $\Omega$ and $\mu$ are as in the previous section.

\begin{definition}
  Let $F\in L_{1,\loc}(\Omega;\K^n)$, $g\in L_{1}(\Omega,\mu)$. Then $g$ is 
called \emph{distributional divergence} of $F$ with respect to $\mu$, denoted
 by $\div_\mu F= g$, if
  \[\int_\Omega F(x) \grad \varphi(x)\, dx = - \int_\Omega g(x) \varphi(x)\, d\mu(x) \quad(\varphi\in C_c^\infty(\Omega)).\]
\end{definition}

\begin{theorem}
\label{thm:operator}
  We have
  \begin{align*}
    D(H) & = \set{u\in D(\form_D);\; \div_\mu \grad \iota(u)\in L_2(\Omega,\mu)},\\
    Hu & = - \div_\mu \grad \iota(u) \quad(u\in D(H)).
  \end{align*}
\end{theorem}

\begin{proof}
  First note that for $u\in D(\form_D)$ and $\varphi \in C_c^\infty(\Omega)$ we have
  \[\form_0(\iota(u),\varphi) = \int_\Omega \grad\iota(u)\cdot \overline{\grad\varphi} = \int_\Omega \grad\iota(u)\cdot \overline{\grad\iota(\varphi)} = \form_D(u,\varphi).\]
  Indeed, since $\iota(\varphi) = J\varphi$ and $\varphi-J\varphi\in W_{2,0}^1(U)$, we obtain
  \[\int_\Omega \grad\iota(u) \cdot \overline{\grad (\varphi-J\varphi)} = 0.\]
  
  Let $H_1$ be the operator defined by the right-hand side in the theorem.
  Let $u\in D(H_1)$ and $\varphi\in C_c^\infty(\Omega)$. Then by the above we have
  \[\form_D(u,\varphi) = \int_\Omega \grad\iota(u)\cdot \overline{\grad\varphi} = -\int_\Omega \div_\mu \grad\iota(u) \overline{\varphi}\, d\mu = \sp{H_1 u}{\varphi}.\]
  By continuity we obtain
  \[\sp{H_1 u}{v} = \form_D(u,v) \quad(v\in D(\form_D)).\]
  Thus, $u\in D(H)$ and $Hu = H_1u$.
  
  To show the converse inclusion let $u\in D(H)\subseteq D(\form_D)$ and $\varphi\in C_c^\infty(\Omega)$. Then
  \begin{align*}
    \int_\Omega \grad\iota(u)\cdot \overline{\grad \varphi} & = \int_\Omega \grad\iota(u)\cdot \overline{\grad \iota(\varphi)} = \form_D(u,\varphi) = \sp{Hu}{\varphi} \\
    & = \int_\Omega Hu \overline{\varphi}\, d\mu.
  \end{align*}
  Hence, $\div_\mu \grad \iota(u)$ exists and $\div_\mu \grad \iota(u) = - Hu\in L_2(\Omega,\mu)$. Thus, $u\in D(H_1)$ and $H_1 u = Hu$.  
\end{proof}

\begin{remark}
 The operator $H$ is the multidimensional analogue of the operator 
$-\partial_\mu\partial \iota $ with Dirichlet boundary conditions, see 
\cite{kkvw09,Seifert2009,SeifertVoigt2011} and also 
\cite{Freiberg2003,Freiberg20042005}.
\end{remark}

We now focus on properties of the semigroup $(e^{-tH})_{t\geq0}$. 
A $C_0$-semigroup $T\from [0,\infty)\to L(L_2(\mu))$ of bounded linear operators in $L_2(\mu)$ is called \emph{positive}, if $T(t)f\geq0$ for all $0\leq f\in L_2(\mu)$, $t\geq0$.
The semigroup is called \emph{submarkovian}, if it is positive and $L_\infty$-contractive, i.e., $f\in L_2(\mu)$, $0\leq f\leq1$ implies $0\leq T(t)f\leq 1$ for all $t\geq 0$.

\begin{theorem}
  The $C_0$-semigroup $(e^{-tH})_{t\geq0}$ is submarkovian.
\end{theorem}

\begin{proof}
  We have to check the Beurling-Deny criteria for the corresponding Dirichlet form $\form_D$. Note that it suffices to check it with $C^\infty$-normal contractions.
  Let $F$ be a $C^\infty$-normal contraction and $u\in C_c^\infty(\Omega)$. Then $F\circ u \in C_c^\infty(\Omega)\subseteq D(\form_D)$. 
  
  We show that $J(F\circ u) = J(F\circ Ju)$.
  Indeed, since $\widetilde{F\circ u} = F\circ u = F\circ\widetilde{Ju} = \widetilde{F\circ Ju}$ $\mu$-a.e. we obtain $F\circ u - F\circ(Ju)\in W_{2,0}^1(U)$. Thus,
  \[0 = J\bigl(F\circ u - F\circ(Ju)\bigr) = J(F\circ u) - J(F\circ(Ju)).\]
  
  We now compute
  \begin{align*}
    \form_D(F\circ u) & = \form_0(\iota(F\circ u)) = \form_0(J(F\circ u)) = \form_0(J(F\circ(Ju))) \\
    & \leq \form_0(F\circ (Ju)) \leq \form_0(Ju) = \form_0(\iota(u)) = 
\form_D(u). \qedhere
  \end{align*}
\end{proof}

\begin{remark}
  In \cite[Section 6.2]{FukushimaOshimaTakeda1994} the traces of Dirichlet forms and associated processes were considered.
  Our result characterizes the corresponding generating operator $H$ in case of (suitably scaled) Brownian motion on a bounded domain $\spt\mu$, 
  where $\mu$ is the corresponding volume measure (i.e.\ Lebesgue measure). The process may jump through $\Omega\setminus \spt\mu$, however (due to the Dirichlet boundary condition at $\partial\Omega$) gets killed on $\partial\Omega$.
\end{remark}

\section{Applications}
\label{sec:applications}

We will now show two applications. Note that by Remark \ref{rem:general} in 
fact we only need to prove $\mu\in M_0(\Omega)$. However, we will also show  
``$\supseteq$'' in \eqref{eq:assumption} (so that equality in 
\eqref{eq:assumption} holds).

Note that for an open subset $V\subseteq \R^n$ we have
\[W_{2,0}^1(V) = \set{u|_V;\; u\in W_{2}^1(\R^n),\, \tilde{u} = 0\,\text{q.e.\ on } \partial V},\]
see e.g.\ \cite[Theorem 2.5]{Frehse1982} and \cite[Theorem 4.2]{BrewsterMitreaMitreaMitrea2012}.

\begin{example}
  Let $n\geq 2$, $\Omega:=(-1,1)^n\subseteq \R^n$, 
$\Gamma:=\Omega\cap(\R^{n-1}\times\set{0})$ and $\mu:=\lambda^{n-1}(\cdot\cap
 \Gamma)$ be the $(n-1)$-dimensional Lebesgue measure on $\Gamma$.
  Then $\mu\in M_0(\Omega)$ by \cite[Theorem 4.1]{BrascheExnerKuperinSeba1994}.
 We will show the equality in \eqref{eq:assumption}.
  Write $\Omega_+:=\Omega\cap(\R^{n-1}\times(0,\infty))$ and $\Omega_-:=\Omega\cap(\R^{n-1}\times(-\infty,0))$.
  
    \begin{figure}[htb]
  \centering
    \begin{tikzpicture}
      \draw (-1,-1)--(-1,2)--(2,2)--(2,-1)--(-1,-1);
      \draw (0.5,-0.5) node{$\Omega_-$};
      \draw (0.5,1.5) node{$\Omega_+$};
      \draw[thick] (-1,0.5)--(2,0.5);
      \draw (-0.8,0.7) node{$\Gamma$};
    \end{tikzpicture}
    \caption{The hypercube $\Omega$, divided into two parts $\Omega_+$ and $\Omega_-$ by the hyperplane $\Gamma$.}
   \end{figure}
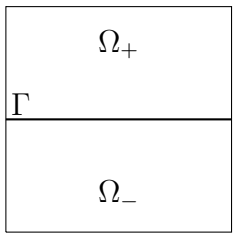
  
  Let $u\in W_{2,0}^1(\Omega)$, $\tilde{u} = 0$ $\mu$-a.e. 
  There exists $(\varphi^k)$ in $C_c^\infty(\Omega)$ such that $\varphi^k\to u$ in $W_2^1(\Omega)$ and $\varphi^k\to \tilde{u}$ q.e.
  Thus, also $\varphi^k(\cdot,0) \to \tilde{u}(\cdot,0) = 0$ $\lambda^{n-1}$-a.e.
  
  For $v\in L_2(\Omega)$ let
  \[Ev(x):=\begin{cases}
	    v(x) & x\in\Omega,\\
	    0 & \R^n\setminus\Omega
           \end{cases}\]
  be the extension of $v$ by zero, and $v_+:=(Ev)|_{\R^{n-1}\times(0,\infty)}$.
  
  We obtain $\varphi^k_+\to u_+$ in $W_{2}^1(\R^{n-1}\times(0,\infty))$. By \cite[Theorem 5.36]{AdamsFournier2003} there exists a bounded linear trace operator $\tr\from W_{2}^1(\R^{n-1}\times(0,\infty))\to L_2(\R^{n-1})$. Hence,
  $\tr\varphi^k_+\to \tr u_+$ in $L_2(\R^{n-1})$. Since also $\tr \varphi^k_+ = \varphi^k(\cdot,0)\to \tilde{u}_+(\cdot,0) = 0$ $\lambda^{n-1}$-a.e.\ we obtain $\tr u_+ = 0$.
  By \cite[Theorem 5.37]{AdamsFournier2003} we obtain $u_+\in W_{2,0}^1(\R^{n-1}\times(0,\infty))$. 
  Two applications of \cite[Theorem 5.29]{AdamsFournier2003} finally yield $u|_{\Omega_+}\in W_{2,0}^1(\Omega_+)$.
  Analogously, $u|_{\Omega_-}\in W_{2,0}^1(\Omega_-)$, and hence $u\in W_{2,0}^1(\Omega)$.  

  Thus the corresponding stochastic process describes a particle diffusing in the hyperplane and jumping through $\Omega$.  
\end{example}

\begin{example}
  Let $D$ be the filled (open) Koch's snowflake centered at the origin and 
$\Omega\subseteq \R^2$ be a large open square centered at the origin such 
that $\overline{D}\subseteq \Omega$.
  Let $\mu:=\lambda^2(\cdot\cap D)$ be the Lebesgue measure on $D$.
  
  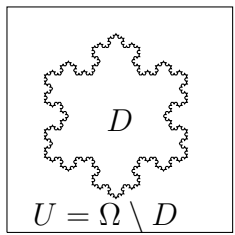
\begin{figure}[htb]
  \centering
  \usetikzlibrary{lindenmayersystems}
  \pgfdeclarelindenmayersystem{Koch curve}{\rule{F -> F-F++F-F}}
  \begin{tikzpicture}
    \draw (-1,-1)--(-1,2)--(2,2)--(2,-1)--(-1,-1);
    \begin{scope}[xshift=-0.5cm, yshift=0cm]
    \draw[scale=1/3] [l-system={Koch curve, step=2pt, angle=60, axiom=F++F++F, order=4}] lindenmayer system -- cycle;     
    \end{scope}
    \draw (-0.8,-0.8) node[right]{$U = \Omega\setminus D$};
    \draw (0.5,0.5) node{$D$};
    
  \end{tikzpicture}
  \caption{The square $\Omega$ and the snowflake $D$.}
  \end{figure}

  Then $\mu\in M_0(\Omega)$. We show equality in \eqref{eq:assumption}.
  Let $u\in W_{2,0}^1(\Omega)$, $\tilde{u} = 0$ $\mu$-a.e. 
  By \cite[Theorem 5.29]{AdamsFournier2003} the extension of $u$ by zero yields
 $u\in W_2^1(\R^2)$.
  By \cite[Theorem 2.5]{Frehse1982} we observe $\tilde{u}=0$ q.e.\ on $\partial\Omega$.
  
  Since $u|_D = 0$ $\lambda^2$-a.e., we have $\tr (u|_D) = 0$ 
$\mathcal{H}^d$-a.e.\ on the boundary of $D$ by \cite[Theorem 2]{Wallin1991},
 where $\mathcal{H}^d$ is the $d$-dimensional Hausdorff measure with 
$d=\frac{\log 4}{\log 3}$.
  By \cite[Corollary 4.5]{BrewsterMitreaMitreaMitrea2012} we thus obtain $\tilde{u} = 0$ q.e.\ on $\partial D$.
  
  Hence, for $U:=\Omega\setminus D$ we obtain $\tilde{u} = 0$ q.e.\ on $\partial U$, which by \cite[Theorem 2.5]{Frehse1982} yields $u\in W_{2,0}^1(U)$.
  
  We can thus describe jump-diffusion, where the diffusion takes part on the snowflake $D$ and jumps may occur along its boundary $\partial D$.
\end{example}

\section*{Acknowledgement}

C.S.\ warmly thanks J\"urgen Voigt and Hendrik Vogt who led him to the topic.

%
%
%
%
%

\bigskip

\noindent
Uta Renata Freiberg \\
Universit\"at Stuttgart \\
Fachbereich Mathematik \\
Institut f\"ur Stochastik und Anwendungen \\
Pfaffenwaldring 57 \\
70569 Stuttgart, Germany \\
{\tt uta.frei\rlap{\textcolor{white}{hugo@egon}}berg@mathematik.uni-\rlap{\textcolor{white}{darmstadt}}stuttgart.de} \\[1em]
\noindent
Christian Seifert\\
Technische Universit\"at Hamburg-Harburg\\
Institut f\"ur Mathematik \\
Schwarzenbergstra{\ss}e 95 E \\
21073 Hamburg, Germany \\
{\tt christian.se\rlap{\textcolor{white}{hugo@egon}}ifert@tuhh\rlap{\textcolor{white}{darmstadt}}.de}

\end{document}